\documentclass{article}
\usepackage{times}
\usepackage{amsmath, amssymb, amsthm, amscd}
\theoremstyle{plain}
\newtheorem{theorem}{Theorem}[section]
\newtheorem{corollary}[theorem]{Corollary}
\newtheorem{definition}[theorem]{Definition}

\begin{document}

\title{Signed-Bit Representations of Real Numbers}

%
\author{Robert S. Lubarsky\\Fred Richman\\
Florida Atlantic University\\rlubarsk@fau.edu\\richman@fau.edu}
\maketitle

\begin{abstract}
The signed-bit representation of real numbers is like the binary
representation, but in addition to 0 and 1 you can also use $-1$.
It lends itself especially well to the constructive
(intuitionistic) theory of the real numbers. The first part of the
paper develops and studies the signed-bit equivalents of three
common notions of a real number: Dedekind cuts, Cauchy sequences,
and regular sequences. This theory is then applied to
homomorphisms of Riesz spaces into $\mathbb{R}$.
\end{abstract}

\section{Introduction}

In \cite{CS}, Coquand and Spitters studied the Stone-Yosida
representation theorem for lattice ordered vector spaces (Riesz
spaces). They gave a constructive proof of this theorem for
separable, seminormed Riesz spaces which used Dependent Choice
(DC). They then asked whether DC is necessary and suggested a
construction which would show that it was. This question was
answered in \cite{LR} and \cite{LR'} along the lines they
suggested.

In thinking about this question, we were led to representing
real numbers in a tree-like structure. This representation
is a lot like the classical signed-bit representation, a
modification of the binary representation where $-1$ is allowed as
well as 0 and 1.
The signed-bit representation is especially suitable to constructivism
and computability because you can show constructively (with DC) that
every real number has a signed-bit representation, but not that every
real number has a binary representation.

The thrust of this paper (Section 2) is this signed-bit
representation. In Sections 3 and 4, the representation is applied to various
questions about real numbers and about homomorphisms
of Riesz spaces into $\mathbb{R}$. The benefits of these
applications include a reformulation of the choice principles
involved, a generalization from countable and separable Riesz
space to ones of arbitrary size, and a recasting of the issues in
a form more familiar to classical set theorists.

\section{Signed-bit representations of real numbers}

\subsection{Three kinds of real numbers}

We are interested in studying real numbers from a constructive
point of view without using countable choice principles. We
consider three kinds of real numbers: Dedekind, regular, and
Cauchy (see also \cite{FH} and \cite{L}). The latter two kinds are
given by sequences of rational numbers (see below). A \emph{real
number}, simpliciter, is a Dedekind real number, that is, a real
number is determined by a located Dedekind cut \cite[Problem
2.6]{B}, \cite[p. 170]{TvD}. A \emph{located Dedekind cut} can be
defined as a nonempty proper open subset $L$ of the rational
numbers $\mathbb{Q}$ such that for all pairs of rational numbers
$u<$ $v$, either $u\in L$ or $v\notin L$. If $r$ is the real
number defined by $L$, then $L=\left\{  u\in\mathbb{Q}:u<r\right\}
$. The Dedekind real numbers are exactly the things that can be
approximated coherently by rational numbers.

If $r$ is any real number, then for each positive integer $n$ there is a
rational number $u$ such that $\left\vert u-r\right\vert \leq1/n$. Using
countable choice, we could construct a sequence $q$ of rational numbers so
that $\left\vert q_{n}-r\right\vert \leq1/n$. Such a sequence $q$ is a
\emph{regular sequence} in the sense that%
\[
\left\vert q_{m}-q_{n}\right\vert \leq\frac{1}{m}+\frac{1}{n}
\]
for all $m$ and $n$. Note that a regular sequence is a Cauchy
sequence, and we leave it as an exercise to show that every Cauchy
sequence converges to some real number.
Conversely, if a regular sequence $q $ converges to the real
number $r$, then $\left\vert q_{n}-r\right\vert \leq1/n$ for all
$n$. Bishop \cite{B} \emph{defines} a real number to be a regular
sequence of rational numbers.

\begin{theorem}
\label{modulus}Let $q$ be a sequence of rational numbers and $\mu$ a sequence
of positive integers. Then the following two conditions are equivalent

\begin{enumerate}
\item For all $i,j$, if $m\geq\mu_{i}$ and $n\geq\mu_{j}$, then%
\[
\left\vert q_{m}-q_{n}\right\vert \leq\frac{1}{i}+\frac{1}{j}.
\]

\item There is a real number $r$ so that for all $i$, if $m\geq\mu_{i}$,
then
$$
  \left\vert q_{m}-r\right\vert \leq1/i.
$$
\end{enumerate}
\end{theorem}

\begin{proof}
If 1 holds, then $q$ is a Cauchy sequence, hence converges to a real number
$r$. If $m\geq\mu_{i}$, then%
\[
\left\vert q_{m}-q_{n}\right\vert \leq\frac{1}{i}+\frac{1}{j}
\]
whenever $n\geq\mu_{j}$. In particular, this inequality holds for arbitrarily
large values of $n$ and $j$, so $\left\vert q_{m}-r\right\vert \leq1/i$.
Conversely, suppose 2 holds. Then%
\[
\left\vert q_{m}-q_{n}\right\vert \leq\left\vert q_{m}-r\right\vert
+\left\vert r-q_{n}\right\vert \leq\frac{1}{i}+\frac{1}{j}
\]
for all $m\geq\mu_{i}$ and $n\geq\mu_{j}$. \medskip
\end{proof}

\noindent We say that $\mu$ is a \emph{modulus of convergence} for $q$ if
either of the equivalent conditions in Theorem \ref{modulus} hold.

If $q$ is a regular sequence, then it has the modulus of
convergence $\mu _{m}=m$. Conversely, if $\mu$ is a modulus of
convergence for $q$, then the sequence $q_{\mu_{m}}$ is a regular
sequence converging to the limit $r$ of $q$. So a real number $r$
is the limit of a regular sequence of rational numbers if and only
if it is the limit of a sequence of rational numbers that has a
modulus of convergence. We call such a real number a \emph{regular
real number}. Troelstra and van Dalen \cite{TvD} define a Cauchy
real number to be what we are calling here a regular real number.

\begin{theorem}
If $r$ is a regular real number, then every sequence of rational numbers
converging to $r$ has a modulus of convergence.
\end{theorem}

\begin{proof}
Let $q$ be a regular sequence of rational numbers converging to
$r$. Let $p$ be a sequence of rational numbers converging to $r$.
We need to find a modulus of convergence $\mu$ for the sequence
$p$.

Given $m$ we define $\mu_{m}$ as follows. Choose $k$ so that
$\left\vert p_{n}-r\right\vert \leq1/6m$ for all $n\geq k$. So, if
$n\geq k$, we have
\[
\left\vert p_{n}-q_{3m}\right\vert \leq\left\vert
p_{n}-r\right\vert +\left\vert r-q_{3m}\right\vert
\leq\frac{1}{6m}+\frac{1}{3m}\leq\frac{1}{2m}
\]
Let $\mu_{m}\leq k$ be the smallest integer such that%
\[
\left\vert p_{n}-q_{3m}\right\vert \leq\frac{1}{2m}
\]
for $n=\mu_{m},\ldots,k$. Then $\mu_{m}$ is the smallest integer
for which the above inequality holds for all $n\geq\mu_{m}$, so
$\mu_{m}$ does not depend on the choice of $k$.

It remains to show that $\left\vert p_{n}-r\right\vert \leq1/m$
for all $n\geq\mu_{m}$. But, if $n\geq\mu_{m}$, then
\[
\left\vert p_{n}-r\right\vert \leq\left\vert
p_{n}-q_{3m}\right\vert +\left\vert q_{3m}-r\right\vert
\leq\frac{1}{2m}+\frac{1}{3m}\leq\frac{1}{m}
\]
\medskip
\end{proof}

\noindent In particular, every sequence of rational numbers that
converges to a rational number has a modulus of convergence.
Irrational numbers are also regular real numbers---in fact, they
have decimal expansions. By an \emph{irrational number} we mean a
real number $r$ such that $\vert r-q\vert > 0$ for each rational
number $q$. It follows that algebraic real numbers, because they
are either rational or irrational, are regular real numbers.

In the absence of countable choice, not every real number can be written as
the limit of a sequence of rational numbers, regular or otherwise. A real
number $r$ that can be so written is called a \emph{Cauchy real number}
because it is the limit of a Cauchy sequence of rational numbers. Not every
Cauchy real number is a regular real number (see \cite{L}).

\subsection{The pseudotree}

We want to consider the following infinite tree-like structure
$T$, the {\em ternary pseudotree}:%

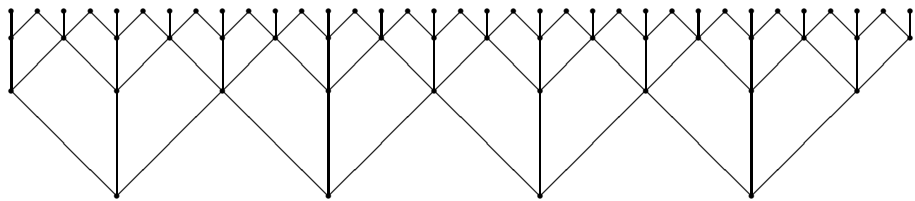
\begin{picture}(400,100)(10,-5)
\put(40,0){\circle*{2}} \put(120,0){\circle*{2}}
\put(200,0){\circle*{2}} \put(280,0){\circle*{2}}
\put(0,40){\circle*{2}}
\put(40,40){\circle*{2}} \put(80,40){\circle*{2}}
\put(120,40){\circle*{2}} \put(160,40){\circle*{2}}
\put(200,40){\circle*{2}} \put(240,40){\circle*{2}}
\put(280,40){\circle*{2}} \put(320,40){\circle*{2}}
\put(0,60){\circle*{2}}
\put(20,60){\circle*{2}} \put(40,60){\circle*{2}}
\put(60,60){\circle*{2}} \put(80,60){\circle*{2}}
\put(100,60){\circle*{2}} \put(120,60){\circle*{2}}
\put(140,60){\circle*{2}} \put(160,60){\circle*{2}}
\put(180,60){\circle*{2}} \put(200,60){\circle*{2}}
\put(220,60){\circle*{2}} \put(240,60){\circle*{2}}
\put(260,60){\circle*{2}} \put(280,60){\circle*{2}}
\put(300,60){\circle*{2}} \put(320,60){\circle*{2}}
\put(340,60){\circle*{2}}
\put(0,70){\circle*{2}} \put(10,70){\circle*{2}}
\put(20,70){\circle*{2}} \put(30,70){\circle*{2}}
\put(40,70){\circle*{2}} \put(50,70){\circle*{2}}
\put(60,70){\circle*{2}} \put(70,70){\circle*{2}}
\put(80,70){\circle*{2}} \put(90,70){\circle*{2}}
\put(100,70){\circle*{2}} \put(110,70){\circle*{2}}
\put(120,70){\circle*{2}} \put(130,70){\circle*{2}}
\put(140,70){\circle*{2}} \put(150,70){\circle*{2}}
\put(160,70){\circle*{2}} \put(170,70){\circle*{2}}
\put(180,70){\circle*{2}} \put(190,70){\circle*{2}}
\put(200,70){\circle*{2}} \put(210,70){\circle*{2}}
\put(220,70){\circle*{2}} \put(230,70){\circle*{2}}
\put(240,70){\circle*{2}} \put(250,70){\circle*{2}}
\put(260,70){\circle*{2}} \put(270,70){\circle*{2}}
\put(280,70){\circle*{2}} \put(290,70){\circle*{2}}
\put(300,70){\circle*{2}} \put(310,70){\circle*{2}}
\put(320,70){\circle*{2}} \put(330,70){\circle*{2}}
\put(340,70){\circle*{2}}
\put(40,0){\line(0,1){40}}
\put(120,0){\line(0,1){40}} \put(200,0){\line(0,1){40}}
\put(280,0){\line(0,1){40}}
\put(40,0){\line(-1,1){40}} \put(120,0){\line(-1,1){40}}
\put(200,0){\line(-1,1){40}} \put(280,0){\line(-1,1){40}}
\put(40,0){\line(1,1){40}}
\put(120,0){\line(1,1){40}} \put(200,0){\line(1,1){40}}
\put(280,0){\line(1,1){40}} \put(0,40){\line(1,1){20}}
\put(40,40){\line(1,1){20}} \put(80,40){\line(1,1){20}}
\put(120,40){\line(1,1){20}} \put(160,40){\line(1,1){20}}
\put(200,40){\line(1,1){20}} \put(240,40){\line(1,1){20}}
\put(280,40){\line(1,1){20}} \put(320,40){\line(1,1){20}}
\put(0,40){\line(0,1){20}} \put(40,40){\line(0,1){20}}
\put(80,40){\line(0,1){20}} \put(120,40){\line(0,1){20}}
\put(160,40){\line(0,1){20}} \put(200,40){\line(0,1){20}}
\put(240,40){\line(0,1){20}} \put(280,40){\line(0,1){20}}
\put(320,40){\line(0,1){20}}
\put(40,40){\line(-1,1){20}} \put(80,40){\line(-1,1){20}}
\put(120,40){\line(-1,1){20}} \put(160,40){\line(-1,1){20}}
\put(200,40){\line(-1,1){20}} \put(240,40){\line(-1,1){20}}
\put(280,40){\line(-1,1){20}} \put(320,40){\line(-1,1){20}}
\put(0,60){\line(0,1){10}}
\put(20,60){\line(0,1){10}} \put(40,60){\line(0,1){10}}
\put(60,60){\line(0,1){10}} \put(80,60){\line(0,1){10}}
\put(100,60){\line(0,1){10}} \put(120,60){\line(0,1){10}}
\put(140,60){\line(0,1){10}} \put(160,60){\line(0,1){10}}
\put(180,60){\line(0,1){10}} \put(200,60){\line(0,1){10}}
\put(220,60){\line(0,1){10}} \put(240,60){\line(0,1){10}}
\put(260,60){\line(0,1){10}} \put(280,60){\line(0,1){10}}
\put(300,60){\line(0,1){10}} \put(320,60){\line(0,1){10}}
\put(340,60){\line(0,1){10}}
\put(0,60){\line(1,1){10}} \put(20,60){\line(1,1){10}}
\put(40,60){\line(1,1){10}} \put(60,60){\line(1,1){10}}
\put(80,60){\line(1,1){10}} \put(100,60){\line(1,1){10}}
\put(120,60){\line(1,1){10}} \put(140,60){\line(1,1){10}}
\put(160,60){\line(1,1){10}} \put(180,60){\line(1,1){10}}
\put(200,60){\line(1,1){10}} \put(220,60){\line(1,1){10}}
\put(240,60){\line(1,1){10}} \put(260,60){\line(1,1){10}}
\put(280,60){\line(1,1){10}} \put(300,60){\line(1,1){10}}
\put(320,60){\line(1,1){10}}
\put(20,60){\line(-1,1){10}} \put(40,60){\line(-1,1){10}}
\put(60,60){\line(-1,1){10}} \put(80,60){\line(-1,1){10}}
\put(100,60){\line(-1,1){10}} \put(120,60){\line(-1,1){10}}
\put(140,60){\line(-1,1){10}} \put(160,60){\line(-1,1){10}}
\put(180,60){\line(-1,1){10}} \put(200,60){\line(-1,1){10}}
\put(220,60){\line(-1,1){10}} \put(240,60){\line(-1,1){10}}
\put(260,60){\line(-1,1){10}} \put(280,60){\line(-1,1){10}}
\put(300,60){\line(-1,1){10}} \put(320,60){\line(-1,1){10}}
\put(340,60){\line(-1,1){10}}
\end{picture}%

\noindent The structure continues infinitely far in all directions
(left, right, up, and down). The nodes are dyadic intervals
$(k/2^{n},(k+2)/2^n)$ where
$k$ and $n$ are integers. The descendants of a node are its subintervals.
For example, the bottom four nodes in the figure could be the intervals
$(-1,0)$, $(-1/2,1/2)$, $(0,1)$, and $(1/2,3/2)$.
The children (immediate descendants) of the node $(0,1)$ are
$(0,1/2)$, $(1/4,3/4)$, and $(1/2,1)$.

The level of a node corresponds inversely to its radius. For
instance, $(0,1)$ is on level $1$ because it has a radius of
$2^{-1}$. In general, the nodes on level $l$ are those with
radius $2^{-l}$, and (hence) length $2^{1-l}$.

A path through $T$ corresponds exactly to a signed-bit
representation of a real number.\footnote {Apparently the first
use of the ternary pseudotree for the signed-bit representation is
in \cite{ABH}. There $T$ is called the {\em Stern-Brocot} or {\em
Farey tree}, even though we find enough difference between each of
those trees and $T$ to warrant the use of a different name. For
more on signed-bit representations themselves, see \cite{W}.} Just
as a number written in binary is a sequence of 0s and 1s, indexed by
$\mathbb{Z}$, in which all entries below some index $n$ are 0, a
signed-bit number, also known as a signed-binary or
signed-digit number, is such a $\mathbb{Z}$-indexed sequence of
0s, 1s, and $-1$s. The sequence $a$ represents the number
$\sum_i a_i 2^{-i}$. No number has a unique representation. The
corresponding path in $T$ starts at the node of length 2$^{n+2}$
with midpoint 0. At stage $i$ the path goes left, middle, or right,
depending on whether $a_i$ is $-1$, 0, or 1 respectively.
Actually, the only paths generated in this way are those that start at
some node with midpoint 0. Those with no such start, or no start
at all, would not correspond to a signed-bit representation in the
sense described here.

If $I$ is a node, we denote the three children of $I$ by $\lambda
I$, $\mu I$, and $\rho I$ (left, middle, and right). An
\emph{extreme descendant} of $I$ is a node of the form
$\lambda^{i}I$ or $\rho^{i}I$ for some $i$.

\subsection{Ideals in $T$ and their real numbers}

Given a real number $r$, let $O_r$ be $\{ I \in T \mid r \in I
\}$, the set of nodes in $T$ that contain $r$. Note that $O_r$ is
closed downwards (under superset) and closed under join
(intersection). An \emph{o-ideal} is a nonempty set $O$ of nodes
closed downwards and under join, such that every node in $O$ has a
nonextreme descendant in $O$.

\begin{theorem} The function $r \mapsto O_r$ is a bijection from
the real numbers to the o-ideals. \end{theorem}

\begin{proof}
To prove that $O_r$ is an o-ideal, we must show that each node of
$O_{r}$ has a nonextreme descendant in $O_{r}$. Suppose
$((k-1)/2^{n}, (k+1)/2^n)\in O_{r}$, that is
\[
\frac{k-1}{2^{n}}<r<\frac{k+1}{2^{n}}
\]
Then there exists $k^{\prime}$ and $n^{\prime}$ such that%
\[
\frac{k-1}{2^{n}}<\frac{k^{\prime}-1}{2^{n^{\prime}}}<r<\frac{k^{\prime}%
+1}{2^{n^{\prime}}}<\frac{k+1}{2^{n}}
\]
But this makes $((k^{\prime}-1)/2^{n^{\prime}},
(k^{\prime}+1)/2^{n^{\prime}})$ a nonextreme descendant of
$((k-1)/2^{n}, (k+1)/2^n)$ in $O_{r}$. Indeed, for
$k^{\prime}/2^{n^{\prime}}$ to be the midpoint of an extreme
descendant, it must be of the form%
\[
\left(  2^{i}(k-1)+1\right)  /2^{n+i}\text{ or }\left(  2^{i}(k+1)-1\right)
/2^{n+i}
\]
so $k^{\prime}=2^{n^{\prime}-n}\left(  k\mp1\right)  \pm1$. But%
\[
2^{n^{\prime}-n}\left(  k+1\right)  -1>k^{\prime}>2^{n^{\prime}-n}\left(
k-1\right)  +1
\]

To see that the function is a bijection, let $O$ be an o-ideal.
Then $O$ defines a set of nonempty open intervals closed under
finite intersection and containing arbitrarily small intervals. So
there is a unique real number $r$ that is contained in all the
closures of intervals in $O$. But because each open interval $J$
in $O$ has a nonextreme descendant, the number $r$ is contained in
$J$ itself. To see that $O=O_{r}$, suppose some dyadic open
interval $J$ contains $r$. Then every sufficiently small dyadic
interval that contains $r$ is contained in $J$. As $O$ is a downset,
$J$ must be in $O$. \medskip
\end{proof}

We can also consider the closed interval correlates. For a real number
$r$, let $C_r$ be $\{ I \in T \mid r \in \overline{I} \}$, where
$\overline{I}$ is the (topological) closure of $I$. The subset $C_r$ is
not closed under join, but it does satisfy the following closure conditions:
\begin{enumerate}
\item Each node in $C_r$ has a child in $C_r$.

\item The nodes at each level in $C_r$ are adjacent, and there are at most
three of them.

\item $\lnot I\notin C_r\Rightarrow I\in C_r$.

\item If $I$ is a node in $C_r$, then $\lambda I\notin C_r\Rightarrow\rho
I\in C_r$, and $\rho I\notin C_r\Rightarrow\lambda I\in C_r$. (By
property 3, these are equivalent.)

\item If $\rho^{i}I\in C_r$ for all $i$, then $I$ is the leftmost member of
three adjacent nodes in the downset, and conversely. Same with
$\rho$ replaced by $\lambda$ and \textquotedblleft
leftmost\textquotedblright\ by \textquotedblleft
rightmost\textquotedblright.

\item If two nodes of $C_r$ have a join in $T$, then that join is in $C_r$.
\end{enumerate}

A \emph{c-ideal} is a nonempty set of nodes satisfying the six
conditions above.

\begin{theorem} The function $r\mapsto C_r$ is a bijection from
the real numbers to the c-ideals. \end{theorem}

\begin{proof}
We first show that $C_{r}$ is a c-downset. Clearly 1, 2, and 6
hold. Property 3 holds because if $J$ is a closed interval, then
$r\in J$ if and only if $\lnot d\left(  r,J\right)  >0$. To see 4,
note that if $r\in J$, but $r\notin \lambda J$, then $r\in \rho
J$, and vice versa. For 5, note that if $r\in \rho^{i}J$ for all
$i$, then $r$ is the right endpoint of $J$.

Now suppose that $C$ is a c-downset. We first show that the
intersection of the intervals $J \in C$ is equal to $\left\{
r\right\}  $ for some real number $r$. Since from 1 there are
arbitrarily small intervals in $C$, it suffices to check the
finite intersection property. So let $F$ be a finite set of nodes
of $C$. If there is a node $J$ in $C$ above all these nodes, then
$J$ is contained in $I$ for all $I \in F$, so the intersection is
nonempty. Otherwise, by 6, there are two nodes in $F$ with no join
in $T$. By 2 this can only happen if there are three adjacent
nodes in $C$, in which case there is a dyadic rational in all the
intervals corresponding to nodes of $I$.

We want to show that $C=C_{r}$. As $r\in J$ for every $J \in C$,
we have $C \subset C_{r}$. We must show that if $r\in J$, then
$J\in C$. By 3 it suffices to assume $J\notin C$ and derive a
contradiction. There is some node $I$ at the level of $J$ that is
in $C$. So $I\neq J$, by the assumption, and also $r\in I$. If the
node $I$ is not next to $J$, then $r$ is the dyadic rational which
is the common endpoint of $I$ and $J$ . This contradicts 5: all
the children of $I$ in $C$ must lean toward $J$ because they all
contain $r$, so by 5 there are three adjacent nodes in $C$. So $I$
and $J$ are next to each other. Similarly, if $I$'s other neighbor
$K$ were in $C$, then all of $K$'s children must lean toward $J$,
contradicting 5. By the adjacency of the nodes in $C$ (property 2)
$I$ is the only node in $C$ at that level. But that also can't
happen: If $\lambda I \in C$ then $I$'s left neighbor is in $C$ by
downward closure, so $\lambda I \notin C$. Symmetrically, $\rho I
\notin C$. By 4, both $\rho I$ and $\lambda I$ are in $C$, the
final contradiction.

Since every c-downset is of the form $C_r$, the function is onto.
It's one-to-one, because if $r \not = r'$ then $C_r \not =
C_{r'}$. \medskip
\end{proof}

If $r$ is a real number, then the infinite paths in $C_{r}$
correspond exactly to the \emph{signed-bit representations} of
$r$. Of course we may not be able to find any such path in the
absence of choice. With choice, property 1 guarantees that every
node of $C_{r}$ is contained in some infinite path. The midpoints
of the nodes of an infinite path in $C_{r}$ form a sequence which
is exactly what Heyting \cite{H} calls a
\emph{canonical number-generator}, so we see that the latter is
essentially a signed-digit representation.

\begin{theorem} \label {modreal} For each real number $r$, the following
are equivalent:
\begin{enumerate}
\item  O$_r$ is countable \item $O_r$ contains an
infinite path \item $C_r$ contains an infinite path \item r is a
regular real number.
\end{enumerate}
\end{theorem}

\begin{proof}
1) implies 2): Starting from any node in $O_r$, taking the first
child and first parent in the counting of $O_r$ produces an
infinite path.

2) implies 3): $O_{r}\subset C_{r}$.

3) implies 4): The midpoints of the intervals of any infinite path
in $C_{r}$ form a regular sequence converging to $r$.

4) implies 1): Let $J$ be some node in $O_r$. Let $J_n$ be a
counting of $J$'s siblings and their descendants  such that each
node occurs infinitely often. Let $c_m$ be a sequence of rational
numbers so that $\left\vert c_{m}-r\right\vert \leq1/m$. At stage
$i$, let $s_i=J_i$ if the closed interval $\left[  c_{i}-1/i,\
c_{i}+1/i\right]  $ is contained in $J_i$, undefined otherwise.
This gives a function $s$ from a detachable subset of $\mathbb{N}$
onto that part of $O_r$ at $J$'s level and beyond, so the latter
is countable by definition. It is easy to alter that counting to
include their ancestors too.
\medskip
\end{proof}

Note that the conditions in Theorem \ref{modreal} are not equivalent to
$C_r$'s being countable:

\begin{theorem}
\label{Kripke}If $C_{r}$ is countable for all regular
real numbers $r$, then for each binary sequence $\alpha$, there
exists a binary sequence $\beta$ such that $\alpha_{m}=0$ for all
$m$ if and only if $\beta_{m}=1$ for some $m$.
\end{theorem}
\begin{proof}
Let $\alpha$ be a binary sequence and set
$r=\sum\alpha_{m}/2^{m}$. Let $C_{r}=\left\{
c_{1},c_{2},c_{3},\ldots\right\}  $. Define $\beta_{m}=1$ if
$c_{m}=(-1,0)$, and $\beta_{m}=0$ otherwise. Note $a_{m}=0$ for
all $m$ if and only if $r=0$. If $r=0$, then $(-1,0)\in C_{r}$ so
$\beta_{m}=1$ for some $m$. Conversely, if $(-1,0)\in C_{r}$, then
$r\leq0$, hence $r=0$. \medskip
\end{proof}

The conclusion of Theorem \ref{Kripke} is a form of the weak
Kripke schema \cite[p. 241]{TvD}. This conclusion, together with
MP (Markov's Principle), implies LPO (the limited principle of
omniscience): any binary sequence $\alpha$ either contains a one
or is all zeros. Indeed, because the sequence $\alpha+\beta$
cannot be all zeros, by MP it must contain a nonzero element
$\alpha_m + \beta_m$; if $\alpha_m = 1$ than $\alpha$ contains a
1, and if $\beta_m = 1$ then $\alpha$ is all 0s. Since MP holds
in the recursive interpretation of constructive mathematics, the
conclusion of Theorem \ref{Kripke} would imply the solvability of
the halting problem. Hence in the recursive interpretation the
conditions of Theorem \ref{modreal} are not equivalent to $C_r$'s being
countable. It would be nice to have a clean characterization of
those real numbers $r$ for which $C_r$ is countable.

For arbitrary Cauchy real numbers the situation is a bit more complicated.
We say that a subset $S$ of $T$ is a \emph{Cauchy subset} if it is
closed downwards, contains nodes from arbitrarily high levels, and
for all $p$ there is a level $l$ such that
$\vert j/2^s - k/2^t \vert < 2^{-p}$
for all nodes $(j/2^s, (j+2)/2^s)$ and
$(k/2^t, (k+2)/2^t)$ beyond $l$ in $S$.
The first clause in that definition says that $S$ is a
downset, the second that $S$ is unbounded. The last says that $S$
converges: given $p$ and $l$ as in the last clause, and $(j/2^s,
(j+2)/2^s)$ with $s>l$, then $(j+1)/2^s$ is within $2^{-p}+2^{-l}$ of
the limit of $S$. So a Cauchy subset is an unbounded, convergent
downset.

Examples of Cauchy subsets $S$ of $T$ are $O_r$ and $C_r$. More
generally, $S$ might also contain bounded branches or subsets that
peter out at a certain point.

It is not hard to see that $O_r \subseteq S$, for the real number $r$
to which $S$ converges. Hence $O_r$ is the
intersection of all the Cauchy subsets converging to $r$. As for
$C_r$, say that a subset $S$ of $T$ is \emph{unblocked} if every
node in $S$ has a child in $S$. Both $O_r$ and $C_r$ are
unblocked. We can characterize $C_r$ as the biggest unblocked
Cauchy subset that converges to $r$.
\begin{theorem}
Any unblocked Cauchy subset of $T$ that converges to $r$ is
contained in $C_r$. So $C_r$ is the union of all unblocked Cauchy
subsets that converge to $r$.
\end{theorem}

\begin{proof}
Let $S$ be an unblocked subset that converges to $r$, let $I \in
S$. We must show that $r\in \overline{I}$. As $S$ is unblocked,
$I$ has descendants -- subsets -- at every level beyond $I$'s and
these get arbitrarily close to $r$. Thus there are elements in $I$
that are arbitrarily close to $r$. As $\overline{I}$ is closed,
this means that $r\in \overline{I}$.
\end{proof}

As for the Cauchy real numbers themselves:

\begin{theorem} A real number $r$ is a Cauchy real number if and only if
$O_r$ is contained in a countable Cauchy subset of $T$. \end{theorem}

\begin{proof}
Suppose $r$ is a Cauchy real number, say the limit of the sequence of
rational numbers $c_n$. Let $J_n= (k/2^n, (k+2)/2^n)$ where $k$ is
the greatest integer such that $k/2^n\leq c_n$. Then $J_n$ is a
node at level $n$ in $T$, and the sequence $J_n$ converges to $r$.
Let $S$ be the downset generated by the terms in the sequence
$J_n$. Conversely, suppose $O_r$ is contained in a countable
Cauchy subset $S$. Then $S$ converges to $r$ and if we let $c_n$
be the midpoint of the first element of $S$ at level $n$, then
$c_n$ converges to $r$.
\medskip
\end{proof}

\section{Choice principles}

We have looked at three kinds of real numbers: Dedekind real
numbers, Cauchy real numbers, and regular real numbers. It is easy
to see that with Countable Choice we can show that these are the
same: we can build a Cauchy sequence from a Dedekind cut by
countably many choices of rationals, and we can build a modulus of
convergence for a Cauchy sequence, by making an appropriate
countable sequence of choices of integers. In fact, since the
choices made are either of a rational number or an integer, we
need only make countably many choices from a countable set, an
axiom variously called AC-NN, AC$_{00}$, and AC$_{\omega\omega}$.
In fact, we can get by on even less:

\begin{theorem}
\label{one}The following choice principles are equivalent:

\begin{enumerate}
\item AC$_{\omega 2}$: Given a sequence $S_{n}$ of nonempty subsets
of $\left\{0,1\right\}$, there exists a binary sequence $a_{n}$
such that $a_{n}\in S_{n}$.

\item AC$_{\omega b}$ for all $b$: For any positive integer $b$
and sequence $S_n$ of nonempty subsets of $\left\{
0,\ldots,b-1\right\}  $, there exists a sequence $a_{n}\in\left\{
0,\ldots,b-1\right\}  $ such that $a_{n}\in S_{n}$.

\item Given a sequence $S_{n}$ of nonempty subsets of $\mathbb{Z}$ of
uniformly bounded lengths (diameters), there exists a sequence
$a_n \in\mathbb{Z}$ such that $a_{n}\in S_{n}$.
\end{enumerate}
\end{theorem}

\begin{proof}
To go from 1 to 2, we induct on $b$. Certainly 2 holds for $b=1$.
If $b>1$, let $\varphi:\left\{  0,\ldots,b\right\}
\rightarrow\left\{  0,\ldots
,b-1\right\}  $ be the retraction that takes $b$ to $b-1$. Let $T_{n}%
=\varphi\left(  S_{n}\right)  $. Then we apply induction to get a
sequence $t_{n}\in T_{n}$, and apply 1 to get a sequence
$a_{n}\in\varphi^{-1}\left(t_{n}\right)  $.

The length of a subset $S$ of $\mathbb{Z}$ is bounded by $b$ if
the difference of any two elements of $S$ is at most $b$. To go
from 2 to 3, let $b$ be a bound on the lengths of the $S_{n}$, and
look at the images of $S_{n}$ modulo $b+1$ considered as subsets
of $\left\{  0,\ldots,b\right\}  $. So we get a sequence
$a_{n}\in\left\{  0,\ldots,b\right\}  $ so that each $S_{n}$
contains an element congruent to $a_{n}$ modulo $b+1$. But that
element of $S_{n}$ is unique.

Of course 3 implies 1. \medskip
\end{proof}

Clearly AC$_{\omega \omega}$ implies the properties above. To
refine the matter even more, let AC$_{\omega, <\omega}$ be the
statement that there is a choice function for the sequence $S_n$,
where each $S_n$ is a bounded set of natural numbers, while
perhaps not uniformly so. Then AC$_{\omega \omega}$ implies
AC$_{\omega, <\omega}$, which in turn implies AC$_{\omega 2}$. The
reason we are looking at this is:

\begin{corollary}
AC$_{\omega 2}$ implies that every real number is regular.
\end{corollary}

\begin{proof}
Let $r$ be a real number. We will construct a sequence $a_n$ of
rational numbers such that $\left\vert r-a_{n}\right\vert \leq
1/n$. To this end, let $S_n=\left\{ m\in\mathbb{Z}:\left\vert
r-m/n\right\vert \leq 1 /n\right\}$. Then $S_n$ is nonempty: since
$r$ is real, there is a rational $q$ within $1/2n$ of $r$, meaning
that $r$ is in the open interval $(q - 1/2n, q + 1/2n)$; the
closed interval $[q - 1/2n, q + 1/2n]$ contains either one or two
fractions of the form $m/n$; and the numerator of any such
fraction will be in $S_n$. Also, $S_n$ is of length at most $2$:
suppose $\vert r-j/n \vert, \vert r-k/n \vert \leq 1/n$, with
$j<k$. From the first inequality, $r \in [(j-1)/n, (j+1)/n]$, and
from the second $r \in [(k-1)/n, (k+1)/n]$. Hence those intervals
must overlap, and so $k-1 \leq j+1$, or $k-j \leq 2$.

Applying (the third version of) AC$_{\omega 2}$, we get a sequence
$m_n$; $a_n = m_n/n$ is as desired.
\medskip
\end{proof}

Presumably we could get by with something less than AC$_{\omega
2}$, since it seems unlikely that AC$_{\omega 2}$ would follow
from every Dedekind real number's being a Cauchy real number,
every Cauchy real number's being a regular real number, or
anything similar. On the other hand, some kind of choice is
necessary, as those equivalences are not theorems in IZF (see
\cite{L}). So exactly what choice principles are those statements
about the real numbers equivalent to? Well, they themselves could
be taken as choice principles. Moreover, it might well be that
among all equivalent formulations, those are the simplest, and so
are the best formulations of some weak choice principles. Still,
it might be useful to have different formulations, and the
versions in terms of the pseudotree $T$ follow immediately from
the work of the previous section.

\begin{corollary} Every real number is a Cauchy real number if and only
if every o-ideal of $T$ is contained in a countable Cauchy subset
of $T$.
\end{corollary}

\begin{corollary} Every Cauchy real number is regular if and only if every
countable Cauchy subset of $T$ contains an infinite path.
\end{corollary}

\section{Riesz spaces}

By a \emph{Riesz space} we mean a lattice-ordered vector space $V$
over the rational numbers. We assume that $V$ has a \emph{unit}: a
distinguished element $1$ such that if $x\in V$, then $x\leq n1$
for some natural number $n$. If $V$ is nontrivial, then $q \mapsto
q1$ gives an embedding of the rational numbers into $V$. We will
identify a rational number $q$ with its image $q1$ in $V$ and
write $x < q$ to mean that $x\leq q'$ for some rational number $q'
< q$.

For $x \in V$, let $x^+ = x \vee 0$ and $x^- = -x \vee 0$.
It follows that $x = x^+ - x^-$. Also, let $|x| = x^+ + x^- \geq 0$.
We say that an element $x\in V$ is an \emph{infinitesimal} if
$|x|\leq q1$ for
every positive rational number $q$, and that $V$ is
\emph{archimedean} if its only infinitesimal element is zero. Note
that $\mathbb{R}$ is an archimedean Riesz space.

Although the field of scalars for a Riesz space is usually
taken to be $\mathbb{R}$ rather than $\mathbb{Q}$, the latter
choice results in a more general structure for the purpose
of constructing homomorphisms into $\mathbb{R}$, our
ultimate interest. That's because
any Riesz $\mathbb{Q}$-homomorphism from a Riesz space
over $\mathbb{R}$ into $\mathbb{R}$ is also an
$\mathbb{R}$-homomorphism.

\begin{theorem}
Let $V$ and $W$ be Riesz spaces over $\mathbb{R}$. If $W$ is archimedean,
then any Riesz homomorphism from $V$ to $W$ over $\mathbb{Q}$ is a
homomorphism over $\mathbb{R}$.
\end{theorem}
\begin{proof}
Let $f:V\rightarrow W$ be a Riesz homomorphism over $\mathbb{Q}$.
For $x\in V$ and $r\in \mathbb{R}$ we must show that $f(rx)=rf(x)$.
As $x$ is the difference of two positive elements of $V$, we may
assume that $x\geq 0$, so $f(x)\geq 0$. Let $p$ and $q$ be
arbitrary rational numbers such that $p\leq r\leq q$.
Then $px\leq rx\leq qx$ so $pf(x)\leq f(rx)\leq qf(x)$ and
$pf(x)\leq rf(x)\leq qf(x)$. It follows that
$$
    (p-q)f(x) \leq f(rx) - rf(x) \leq (q-p)f(x)
$$
Because $|q-p|$ can be arbitrarily small, and $W$ is archimedean,
this implies that $f(rx) = rf(x)$.
\end{proof}

We cannot eliminate the condition that $W$ be archimedean from
this theorem because of the following classical counterexample.
Let $V = \mathbb{R}\times \mathbb{R}$ with the lexicographic
order. Note that we cannot find a constructive proof of the
existence of the join of two elements in $V$. Let
$g:\mathbb{R}\rightarrow \mathbb{R}$ be a linear transformation
over $\mathbb{Q}$ and define $f:V\rightarrow V$ by $f(x,y) =
(x,g(x)+y)$. It is easy to see that $f$ is a Riesz homomorphism
over $\mathbb{Q}$, and that $f$ is a homomorphism over
$\mathbb{R}$ if and only if $g$ is a linear transformation over
$\mathbb{R}$.

The canonical example of an archimedean Riesz space
is a space $E$ of bounded real-valued functions on a set $X$ that
contains the constant function $1$. Evaluation at
a point of $X$ is a Riesz homomorphism from $E$ into $\mathbb{R}$.
The set of homomorphisms from a Riesz space to $\mathbb{R}$ has a
natural topology and is often called the {\it spectrum} of the
Riesz space \cite{CS, Fr}.

Conversely, any archimedean Riesz space $V$ can be embedded as a
subspace of the space of real-valued continuous functions on its
spectrum (the Stone-Yosida representation theorem). The embedding
of $V$ takes $a \in V$ to the function $\hat a(\sigma) =
\sigma(a)$. This is why we are interested in homomorphisms of $V$
into $\mathbb{R}$. The standard proofs of the Stone-Yosida theorem
are not constructive as they rely heavily on both the law of
excluded middle and the axiom of choice.

Following \cite{CS}, let
   $U(a) = \{ q \in \mathbb{Q} \mid a < q\}$.
The set $U(a)$ is an upper cut in the rational numbers, but need
not be located, so might not correspond to a real number. Still,
$U(a)$ has many of the  characteristics of a real number (and so
is sometimes called an {\em upper real number}, for instance in
\cite{CS}). For instance, for $p$ rational, we will have need of
the predicates $p \leq U(a)$, which means $p\leq q$ for all $q\in
U(a)$, and $p < U(a)$, which means that $p < q \leq U(a)$ for some
rational number $q$.

If $U(a)$ is located, then it is the upper cut of a (Dedekind)
real number $\sup(a)$. If $U(a)$ is located for every $a \in V$, then
$\sup(|\cdot|)$ is a seminorm on $V$. This will be a norm exactly when
$V$ is archimedean.

If $I$ is the interval $(p,q)$, then we let the string of symbols
``$a \in I$" denote the Riesz space element $(a-p) \wedge (q-a)$. We
will be working with the predicate Pos($a$) = ``$0 < U(a)$", even if
$U(a)$ is not located. Note that if $V$ is a
function space, with 1 the constant function with value 1, then
classically Pos($a \in I$) exactly when $a$ takes on a value in $I$.

We denote the set of functions from $A$ to $B$ by $^AB$. If $B$ is a
partially ordered set, and $f_i \in$ $^{A_i}B$, then we set
$f_1 \leq f_2$ if $A_1 \subseteq A_2$ and $f_1(a) \leq f_2(a)$ for all
$a\in A_1$.

\begin{definition} \label{s-br}
Let $X$ be a set and $\chi$ a set of functions from finite subsets
of $X$ to $T$.
\begin{enumerate}
\item We say that $\chi$ is {\bf well-formed}, and that $X$ is the
domain of $\chi$, if

\begin{itemize}
\item $X = \bigcup_{{\bf I} \in \chi} dom({\bf I})$, and
\item $\chi$ is closed downwards.
\end{itemize}

\item A well-formed $\chi$ is {\bf extendible} if, for all {\bf I}
$\in \chi$, $u \in X$, and $n \in \mathbb{N}$, there is a {\bf J}
$\in \chi$ extending {\bf I} with $u \in$ dom({\bf J}) and
level($J_u$) $\geq n$.

\item Let $X$ be a subset of a Riesz space $V$. The {\bf signed-bit
representation} of $X$, with notation $X_T$, is the subset of
$\; \; \bigcup_Y \; ^YT$, as $Y$ ranges over all finite subsets of
$X$, such that {\bf I} = $(I_y)_{y \in Y} \in X_T$ iff
Pos$(\bigwedge_{y \in Y} y \in I_y)$.
\end{enumerate}
\end{definition}

It is immediate that the signed-bit representation $X_T$ is
well-formed, with domain $X$. The essence of the Coquand-Spitters
construction is that, if $U(a)$ is located for all $a \in V$, then
$V_T$ is also extendible. The way they use this is to build Riesz
homomorphisms of a separable Riesz space $V$ into $\mathbb{R}$
(there called {\it representations}), as follows. They take
$X$ to be a countable dense subset of $V$ and let {\bf I} be any
starting point in $X_T$. Using DC, they then extend {\bf I} to all
levels and to include all of $X$, yielding a homomorphism of $X$,
which, by density, can be extended uniquely to all of $V$.

\begin{definition} An o-ideal through $\chi$ is an assignment of an
o-ideal $r_x$ through $T$ to each $x$ in the domain $X$ of $\chi$
such that, for all {\bf I} = $(I_y)_{y \in Y} \in \Pi_y r_y$, {\bf
I} $\in \chi$.
\end{definition}

\begin{theorem} There is a canonical bijection between Riesz
homomorphisms of $V$ into $\mathbb{R}$ and o-ideals through $V_T$.
\end{theorem}
\begin{proof} By results of the section 2, an o-ideal can
be considered to be a real number. So both homomorphisms of $V$
into $\mathbb{R}$ and o-ideals through $V_T$ are assignments of
real numbers to the members of $V$. The coherence conditions on a
Riesz homomorphism correspond to the positivity predicate in the
definition of the extendible set $V_T$.

The main technical lemma needed is that, if $f$ is such a
homomorphism, and $f(a) > 0$, then Pos($a$). So let $q$ be such
that $f(a) \geq q > 0.$ Suppose $r \in U(a)$. Then $r>a$, and $r =
f(r) \geq f(a) \geq q > 0$, as desired.

In some detail, let $f : V \rightarrow \mathbb{R}$ be a Riesz
homomorphism. The induced o-ideal is given by $x \mapsto
O_{f(x)}$. (Recall that $O_r$ is the o-ideal corresponding to
$r$.) We must show that this is through $V_T$, which means that if
$I_y \in O_{f(y)}$ for each $y$ in a finite set $Y$ then $(I_y)_{y
\in Y} \in V_T$. And that means Pos$(\bigwedge_{y \in Y} y \in
I_y)$. By the lemma, it suffices to show that $f(\bigwedge_{y \in
Y} y \in I_y) > 0$. Because $f$ is a homomorphism, the left-hand
side equals $\bigwedge_{y \in Y} f(y \in I_y)$. The infinimum of a
finite set of real numbers is positive if and only if each of
those reals is positive. So we need to show that $I \in O_{f(y)}$
implies $f(y \in I) > 0$. Recall that $I \in O_r$ iff $r \in I$
iff $\inf I < r < \sup I$. Also recall that $y \in I$ is an
abbreviation for $y - \inf I \wedge \sup I - y$. So what we need
to show is that $\inf I < f(y) < \sup I$ implies $f(y - \inf I
\wedge \sup I - y) > 0$. Again using that $f$ is a homomorphism,
the latter assertion reduces to $f(y) - \inf I > 0$ and $\sup I -
f(y) > 0$, which is exactly the hypothesis.

In the other direction, suppose that $x \mapsto O_{r_x}$ is an
o-ideal through $V_T$. Let $f(x) = r_x$. We must show that $f$ is
a Riesz homomorphism: $f(x+y) = f(x) + f(y), f(rx) = rf(x), f(1)
=1,$ and $f(x \wedge y) = f(x) \wedge f(y)$. We will prove the
first statement, and leave the others, all similar, to the reader.

Given $\epsilon > 0$, let $1/2^n < \epsilon/4$ and $I_x \in
O_{r_x}, I_y \in O_{r_y}$ have length $1/2^n$. Then the interval
$I_x + I_y$ has length less than $\epsilon/2$. We claim that any
$I \in O_{r_{x+y}}$ has to have a non-empty intersection with $I_x
+ I_y$. To this end, let $I \in O_{r_{x+y}}$. Because we're
dealing with intervals with rational endpoints, we can assume that
the intersection is empty and come up with a contradiction. For
the intersection to be empty, either $\inf I \geq \sup(I_x) +
\sup(I_y)$ or $\sup I \leq \inf(I_x) + \inf(I_y)$; we will
consider the former case only. Because the system $O_{r_x}$ is an
o-ideal through $V_T$, we have that the triple $(I_x, I_y, I)$ is
in $V_T$, i.e. Pos$(x \in I_x \wedge y \in I_y \wedge x+y \in I)$.
Unpacking that Riesz space element, we get Pos$(x - \inf(I_x)
\wedge \sup(I_x) - x \wedge y - \inf(I_y) \wedge \sup(I_y) - y
\wedge (x+y) - \inf I \wedge \sup I - (x+y))$. That latter Riesz
space element is less than or equal to $\sup(I_x) - x \wedge
\sup(I_y) - y \wedge (x+y) - \inf I$, which, by the case
hypothesis, is less than or equal to $\sup(I_x) - x \wedge
\sup(I_y) - y \wedge (x+y) - (\sup(I_x) + \sup(I_y))$. This last
element is of the form $e \wedge f \wedge (-e-f)$, which can be
shown by elementary Riesz space considerations to be $\leq 0$, in
other words not Pos$(e \wedge f \wedge (-e-f))$, which is the
desired contradiction.

Now pick an interval $I$ in $O_{r_{x+y}}$ of length less than
$\epsilon/2$. This $I$, which contains $f(x+y)$, overlaps $I_x +
I_y$, which contains $f(x) + f(y)$, so $f(x+y)$ is within
$\epsilon$ of $f(x) + f(y)$.
\end{proof}

So by converting a real number to a substructure of the tree-like
partial order $T$, homomorphisms of $V$ are converted to
substructures of products of $T$. Similar theorems hold for other
natural substructures of $T$.

\begin{definition} An o-ideal through $\chi$ is countable if each
$r_x$ is countable.
\end{definition}

\begin{theorem}
There is a canonical bijection between Riesz homomorphisms of $V$
into the regular real numbers and countable o-ideals through
$V_T$.
\end{theorem}

\begin{definition} An o-ideal through $\chi$ is countably
extendible if each $r_x$ is a subset of a countable Cauchy
subtree of T.
\end{definition}

\begin{theorem}
There is a canonical bijection between Riesz homomorphisms of R
into the Cauchy real numbers and countably extendible o-ideals through
$V_T$.
\end{theorem}

The proofs here are the same as in theorem 4.4, with the
additional observation that, when transforming Riesz homomorphisms
into o-ideals and vice versa, Cauchy reals are taken to Cauchy
reals and regular reals to regular reals.

Similar considerations apply to extending Riesz homomorphisms from
dense subsets. That is, suppose $X$ is a dense subset of a Riesz
space $V$. Then it makes no sense in general to talk about a Riesz
homomorphism of $X$, since $X$ might not even be a Riesz space.
However, $X_T$ contains the nearness information about $V$, so
that an o-ideal through $X_T$ induces a homomorphism of $V$. In
fact, these observations could be combined with those above, so
that $X$ need be taken only as a Riesz generating subset of a
dense set, for instance as the members of a dense set between 0
and 1. Then an o-ideal through $X_T$ is canonically extendible to
the generated Riesz space, which by density could be extended to
one through the whole Riesz space.

When extending homomorphisms this way, you no longer have a choice
of what kind of real numbers to use. That is, when dealing with
only Riesz-space structure (addition, scalar multiplication, and
sup), the corresponding operations on real numbers never take you
outside of any given class of real numbers: the sum of two
countable o-ideals is again countable, as is any multiple or sup
of such, and so on. However, the same no longer applies to limits
when dealing with density. A limit or accumulation point may not
have any countable sequence approaching it, so it should be clear
that attaching a Cauchy sequence, even if regular, to dense many
points in a neighborhood will not necessarily yield a Cauchy
sequence at the given point. Worse yet, even if we had that every
point in $V$ were the limit of a countable sequence from $X$,
there would still be problems going from Cauchy sequences on $X$
to ones on all of $V$: choosing a limiting sequence, choosing a
Cauchy sequence for each point in the sequence, etc. (For similar
issues in the simpler context of the real numbers alone, see
\cite{L}.) So the best we really can say is that any kind of
o-ideal on $X_T$ induces simply an o-ideal on $V_T$, i.e. a Riesz
homomorphism of $V$ into the Dedekind real numbers.

These considerations lead to the following
\begin{theorem}
If every extendible $\chi$ with $X$ of cardinality $\kappa$ has an
o-ideal, then every seminormed Riesz space with a dense subset of
cardinality $\kappa$ has a Riesz homomorphism into $\mathbb{R}$.
\end{theorem}

By cardinality here, we mean simply the Cantorian theory of
equinumerosity. So $\kappa$ is simply a set, and a set $X$
has cardinality $\kappa$ if it can be put into one-to-one
correspondence with $\kappa$. The latter principle has the flavor of a
restricted form of Martin's Axiom, hence the following definition.

\begin{definition} Martin's Axiom for o-ideals of cardinality $\kappa$,
written MA$_{\text{o-id}(\kappa)}$, is the assertion that every
extendible $\chi$ with $X$ of cardinality $\kappa$ has an o-ideal.
\end{definition}

One possible benefit of the reformulation of the existence of such
homomorphisms as MA$_{\text{o-id}(\kappa)}$ is that it can help
show that such homomorphisms do not exist. In \cite{CS}, Coquand
and Spitters show, under DC, that every separable, seminormed $V$
has a Riesz homomorphism into $\mathbb{R}$, essentially by showing
MA$_{\text{o-id}(\omega)}$. Of course, they don't refer to
signed-bit representations, and their definition of
\emph{countable} is broader than ``equinumerous with $\omega$'',
as is standard in constructive analysis (see \cite{B}). They then
ask whether DC is necessary. One way to approach that problem is
to find a model in which MA$_{\text{o-id}(\omega)}$ fails in such
a way that an equivalent Riesz space can be constructed from this
failure. In fact, this project was carried out. It was later
simplified \cite{LR'} to refer not to $T$ and its paths but more
simply to $\mathbb{R}$, which is better understood.

A limitation of the last theorem is that it is not a
biconditional. Indeed, we could not find any equivalence between
well-formed sets, possibly with extra conditions, on the one hand,
and any kind of Riesz spaces on the other. In the current
formulation, for instance, having Riesz homomorphisms into
$\mathbb{R}$ for every Riesz space might not be enough to get
o-ideals through all extendible $\chi$s, because $\chi$ might not
correspond to a Riesz space. Furthermore, there seems to be no
elegant formulation of a well-formed $\chi$ coming from a Riesz
space. One could consider instead all extendible $\chi$s, with
domain $X$, and extend $X$ to a Riesz space $V$ so that the
signed-bit representation of $X$ is exactly $\chi$. The problem
there is guaranteeing that $V$ is seminormed, with again
apparently no nice way of identifying those $\chi$s for which the
induced $V$ is seminormed. One could try to be more general, and
eliminate the restriction of $V$ being seminormed. There are
examples of function spaces that are not seminormed for which the
signed-bit representation is not extendible. You might then think
to eliminate the requirement of extendibility. But then there are
problems representing faithfully partial information about a Riesz
space in a well-formed set. In the end, it remains unclear what
an exact correspondence here would be. It would be interesting to
see such a theorem.

\nocite{ex1,ex2}
\bibliographystyle{latex8}
\bibliography{latex8}
\end{document}